\documentclass[11pt]{article}
\usepackage[active]{srcltx}
\usepackage{amsmath, amssymb, amsthm, bm}
\usepackage{tikz}
\usepackage{setspace}

\usepackage[scale=0.8]{geometry}

\newcommand{\Oh}{\mathcal{O}}
\renewcommand{\Pr}[1]{\ensuremath{\operatorname{\mathbf{Pr}}\left[#1\right]}}

\newcommand{\Ex}[1]{\ensuremath{\operatorname{\mathbf{E}}\left[#1\right]}}
\newcommand{\Var}[1]{\ensuremath{\operatorname{\mathbf{Var}}\left[#1\right]}}
\newcommand{\Vars}[2]{\ensuremath{\operatorname{\mathbf{Var}_{#1}}\left[#2\right]}}
\newcommand{\Exs}[2]{\ensuremath{\operatorname{\mathbf{E}_{#1}}\left[#2\right]}}

\newtheorem{thm}{Theorem}  
\newtheorem*{que}{Question}

\newtheorem{lem}[thm]{Lemma}

\newtheorem{cor}[thm]{Corollary}

\newtheorem{pro}[thm]{Proposition}

\title{Cutoff Phenomenon for Random Walks on Kneser Graphs}

\author{Ali Pourmiri \footnote{Max Planck Institute for Informatics, Saarbr\"ucken, Germany, email: {\tt pourmiri@mpi-inf.mpg.de}} \and Thomas Sauerwald \footnote{University of Cambridge, United Kingdom, email: {\tt thomas.sauerwald@cl.cam.ac.uk}}}
      
\begin{document}

\maketitle

\begin{abstract}
The cutoff phenomenon for an ergodic Markov chain 
describes a sharp transition in the convergence to its stationary distribution, over a negligible period of time, known as  cutoff window. 
 We study the cutoff phenomenon for simple random walks on Kneser graphs, which is a family of ergodic Markov chains. Given two integers $n$ and $k$, the Kneser graph $K(2n+k,n)$ is defined as the graph with vertex set being all subsets of $\{1,\ldots,2n+k\}$ of size $n$ and two vertices $A$ and $B$ being connected by an edge if $A\cap B =\emptyset$.  We show that  for any $k=O(n)$, the random walk on $K(2n+k,n)$ exhibits a cutoff at $\frac{1}{2}\log_{1+k/n}{(2n+k)}$ with a window of size $O(\frac{n}{k})$.
\end{abstract}

\medskip

\newpage

\section{Introduction}
A simple random walk on a finite, non-bipartite graph is a discrete-time ergodic Markov chain, where in each time step the walk, located at some vertex, chooses one of its neighbor uniformly at random and moves to that neighbor. 
The cutoff phenomenon for a sequence of  chains describes a sharp transition in the convergence of the chain distribution to its stationary distribution, over a negligible period of time, known as  cutoff window. For applications such as MCMC a cutoff is desirable, as running the chain any longer than the mixing time becomes essentially redundant. From a theoretical perspective,  establishing a cutoff is often surprisingly challenging, even for simple chains, as it requires very tight bounds on the distribution near the mixing time. 

Let $P$ be a transition matrix of an ergodic (i.e., aperiodic and irreducible), discrete-time Markov chain $(X_0, X_1,\ldots)$ on a finite
state space $\Omega$ with stationary distribution $\pi$.  Let $P^t(x, .)$ be the probability distribution of the chain at time $t\in \mathbb{N}$ with starting state $x\in \Omega$. 
The total variation distance between two probability distributions $\mu$ and $\nu$ on a probability space $\Omega$ is defined by
\begin{align*}
 \| \mu-\nu \|_{TV}=\max_{A\subset\Omega}|\mu(A)-\nu(A)|\in [0, 1].
\end{align*} Therefore, we can define the worst-case total variation distance to stationarity at time $t$ as
\[
d(t)=\max_{x\in \Omega}\| P^t(x, .)-\pi\|_{TV}.
\] 
For convenience, we define $d(t)$ for non-integer $t$ as $d(t):=d(\lfloor t\rfloor)$.
(If the reference is clear from the context, we will also just say total variation distance at time $t$). The mixing time is defined by
\[
 t_{mix}(\epsilon)=\min\{t\in \mathbb{N}: d(t)<\epsilon\}.
\]
Suppose now that we have a sequence of ergodic finite Markov chains indexed by $n=1,2,\ldots$.  Let $d_n(t)$   be the total variation distance of the $n$-th chain at time $t$ and $t^{(n)}_{mix}(\epsilon)$ be its mixing time.
 Formally, 
we say that the sequence of chains exhibits a cutoff (in total variation distance),  as defined in \cite[Section 18.1]{lev},  if for  any fixed $0< \epsilon <1$,
\begin{align*}
\lim_{n\rightarrow\infty} \frac{t^{(n)}_{mix}(\epsilon)}{t^{(n)}_{mix}(1-\epsilon)}=1, 
\end{align*}
or equivalently, a sequence of Markov chains has a cutoff at time $t_n$ with a window of size $w_n=o(t_n(1/4))$ if  
\begin{align}
\lim_{\lambda\rightarrow\infty}\liminf_{n\rightarrow\infty}d_n(t_n-\lambda w_n) &=1,\notag\\
\lim_{\lambda\rightarrow\infty}\limsup_{n\rightarrow\infty}d_n(t_n+\lambda w_n)&=0.\label{eq:two}
\end{align}
Although it is widely believed that many natural families of Markov chains exhibit a cutoff,  there are  relatively  few examples where cutoff has been  shown. It turns out that this is quite challenging  to prove or disprove the existence of a cutoff even for simple family of chains.
The first results exhibiting a cutoff appeared in the studies of card-shuffling processes by Aldous and Diaconis~\cite{Aldous}, and Diaconis and Shahshahani~\cite{Diaconis}. Later, the cutoff phenomenon was also shown for random walks on hypercubes \cite{Diaconis1}, for random walks on distance regular graphs including  Johnson and  Hamming graphs  \cite{Belsley,Diaconis2}, and for randomized riffle shuffles \cite{Chen}.  For a more general view of  Markov chains with and without cutoff we refer the reader to \cite{Diaconis3} or \cite[Chapter 18]{lev}.
A necessary condition, known as product condition,  for a family  of chains to exhibit cutoff  is that $t_{mix}^n(1/4)\cdot {\tt gap}_n$ tends to infinity as  $n$ goes to infinity, where ${\tt gap}_n$ is the spectral gap of the transition matrix of  $n$-th  chain (see \cite[Proposition 18.3]{lev}). 
 However there are some chains where the product condition holds and they do not show any cutoff (e.g see \cite[Section 18]{lev}), Peres \cite{Pe04} conjectured that many  natural family of  chains satisfying the product condition exhibit  cutoffs. For instance, he conjectured that random walks on any family of $n$-vertex (transitive) expander  graphs with ${\tt gap}_n=\Theta(1)$ and mixing time $\Oh(\log n)$ exhibit cutoffs. Chen and Saloff-Coste \cite{CS08} verified the conjecture for other distances like the $\ell^p$-norm for $p>1$.  Recently, Lubetzky and Sly \cite{lub} exhibited cutoff phenomena for random walks on random regular  graphs. They also showed that there exist families of explicit expanders with and without cutoff \cite{lub1}.
 Diaconis \cite{Diaconis3} pointed out that if  the second largest eigenvalues of  the transition matrix of a chain   has high multiplicity, then this chain is more likely to show a cutoff.   

In this work, we focus on simple random walks on Kneser graphs. The Kneser graph is defined as follows. For any two positive integers $n$ and $k$, the Kneser graph $K(2n+k, n)$ is the graph with all $n$-element subsets of $[2n+k]=\{1,2,\ldots,2n+k\}$ as vertices and two vertices adjacent if and only if
their corresponding $n$-element subsets are disjoint. We emphasize that throughout this paper, $k$ and $n$ are arbitrary integers, in particular, $k$ can be a function of $n$.  In the case that $k=\omega(n)$, the number of vertices which is ${2n+k}\choose{n}$ and  degree of each vertex, ${n+k}\choose{n}$, have the same magnitude so  the simple random walk on $K(2n+k,n)$ is  mixed  in just one step. For the special case $k=1$, we obtain the so-called odd graph $K(2n+1,n)$ with large odd cycles of size $2n+1$, which is an induced subgraph of $K(2n+k,n)$.  This proves that  $K(2n+k,n)$ is not bipartite for every $k\geq 1$. The permutation group on $[2n+k]$ is a subgroup of the automorphism group of $K(2n+k,n)$, and thus the Kneser graph is always transitive.
Combining these two observations, we conclude that the simple random walk on  $K(2n+k,n)$ is an ergodic and transitive Markov chain.
 Kneser graphs have been studied frequently in (algebraic) graph theory, in particular due to their connections to chromatic numbers and graph homomorphisms (see \cite{god} for more details and references).

Godsil \cite{god1} shows that for most values of $n$ and $k$, the graph $K(2n+k,n)$ is not a Cayley graph.
%
It is also well-known that   the transition matrix of the simple random walk on Kneser graph $K(2n+k, n)$ has spectral gap $\frac{k}{n+k}$ and its second largest eigenvalue has multiplicity ${2n+k}$ (cf.~Corollary \ref{cor:god}). So by varying $k=\Oh(n)$, we obtain various family of  chains with different spectral gaps.  For instance by setting  $k=\Theta(n)$ we obtain a family of transitive expander graphs. 
In order to show a cutoff for a simple random walk on Kneser graphs it is necessary to have a sufficiently tight estimate of its mixing time.   
Let $P$ be the transition matrix of the  simple random walk on Kneser graph $K(2n+k,n)$ with spectrum
$\lambda_i$, $0\leq i\leq {2n+k \choose n}-1 $ and $\lambda_0=1$. Then it is shown that
\cite[Lemma~12.16]{lev} 
\begin{align}
d(t) &= \max_{x\in \Omega}\| P^t(x, .)-\pi\|_{TV}\leq \frac{1}{2}\sqrt{\sum_{i=1}^{|\Omega|-1}\lambda_i^{2t}}, \label{eq:spec}
\end{align}
 where $\Omega$ is the vertex set of the graph.
 It may be surprising that the upper bound obtained by the spectral properties of transition matrix is sufficiently tight and matches the lower bound, which enables us to show the existence of a  cutoff. 
 Besides Kneser graphs, the bound in (\ref{eq:spec}) has been successfully applied in  computing of the mixing time 
of random walks on Cayley graphs (see \cite{Diaconis4,Dou}). 
 This may suggest the following question:
 \begin{que}
  For which families of transitive ergodic chains is the upper bound  in (\ref{eq:spec}) tight up to low order terms?
 \end{que} 


\section{Result}

In the following we state the main result of the paper.
\begin{thm}\label{thm:main}
The simple random walk on $K(2n+k,n)$ exhibits a cutoff at $\frac{1}{2}\log_{1+k/n} (2n+k) $ with a cutoff window of size $O(\frac{n}{k})$ for  $k=O(n)$.
\end{thm}

We now give the proof of Theorem~\ref{thm:main} using Proposition~\ref{pro:upper} and \ref{pro:lower}, whose statements and proofs are deferred to later sections.

\begin{proof}
  For the proof of the upper bound on the mixing time, we use the spectrum of the transition matrix.
 Applying Proposition~\ref{pro:upper} implies that 
 \[
 \lim_{c\rightarrow\infty}\liminf_{n\rightarrow\infty}d_n \Big(\frac{1}{2}\log_{1+k/n} (2n+k) +c\frac{n}{k} \Big)=0.
 \]
  We establish the lower bound by considering the vertices visited by a random walk starting from $\{n+1,\ldots,2n\}$ and their intersection with $[n]=\{ 1,\ldots, n\}$. For any step, we compute the expected size of the intersection and derive an upper bound on its variance  (to stationarity). Then applying Proposition~\ref{pro:lower}  results into
\begin{align*}
\lim_{c\rightarrow\infty}\liminf_{n\rightarrow\infty}d_n\left(\frac{1}{2}\log_{1+k/n} (2n+k) -c\frac{n}{k}\right)&=0.
\end{align*}
Combining these findings establishes a cutoff at $\frac{1}{2}\log_{1+k/n} (2n+k) $ with a cutoff window of size $O(\frac{n}{k})$ for  $k=O(n)$.
\end{proof}

\section{Upper Bound on the Variation Distance}

To prove our results, we need two  lemmas, the lemma below can be found in~\cite[Lemma~12.16]{lev}.  
\begin{lem}[{\cite[Lemma~12.16]{lev}}]\label{lem:tra}
 Let $P$ be a reversible transition matrix with eigenvalues 

\[
 1=\lambda_0\geq\lambda_1\geq\dots\geq\lambda_{|\Omega|-1}.
\]
If the Markov chain is transitive, then for every $x \in \Omega$
\[
 4 \| P^t(x,.)-\pi \|^2_{TV} \leq \sum_{i=1}^{|\Omega|-1}\lambda^{2t}_i.
\]
\end{lem}

To employ Lemma~\ref{lem:tra}, we need to know all eigenvalues and their multiplicities. The spectrum of the
adjacency matrix of Kneser graphs was computed in \cite[Section~9.4]{god}  and \cite{rein}. 
\begin{thm}[{\cite[Section~9.4]{god}  and \cite{rein}}]\label{thm:god}
The adjacency matrix of Kneser graphs $K(2n+k,n)$ has the following spectrum
\begin{align*}
 (-1)^i{n+k-i \choose n-i} \quad \text{ with multiplicity of~~~${2n+k \choose i}- {2n+k\choose {i-1}}$}, \quad i=0,\ldots,n,
\end{align*}
where ${2n+k\choose {-1}} = 0$. 
\end{thm}
As $K(2n+k,n)$ is a $\binom{n+k}{n}$-regular graph, we immediately obtain the following corollary.
\begin{cor}\label{cor:god}
 The transition matrix of the simple random walk on $K(2n+k,n)$ has the following spectrum:
\[
 (-1)^i\frac{{n+k-i\choose n-i}}{{n+k\choose n}} \quad \text{ with multiplicity of~~~~${2n+k \choose i}- {2n+k\choose {i-1}}$}, \quad i=0,\ldots,n.
\]
 
\end{cor}

\begin{pro}\label{pro:upper}
 We have the following upper bounds on the total variation distance of the simple random walk on $K(2n+k,n)$.
\begin{itemize}
\item If $k= o(n)$, then for every constant $c \geq 1/2$,
\[
      d\left(\frac{1}{2}\log_{1+k/n} (2n+k) +c\frac{n}{k}\right)\leq e^{-c}.
\]

\item If $k = \Omega(n)$, then  for every constant $c$ with $ (1+\frac
{k}{n})^{-c}\leq \frac{1}{2}$,
\[
    d\left(\frac{1}{2}\log _{1+k/n}(2n+k)+c\right)\leq (1+k/n)^{-c}.
\]
 \end{itemize}
\end{pro}

\begin{proof}
By Corollary~\ref{cor:god} we have 
\[
 |\lambda_i|= \left|(-1)^i\frac{n(n-1)(n-2) \cdot \ldots \cdot (n-i+1)}{(n+k)(n+k-1)(n+k-2) \cdot \ldots \cdot (n+k-i+1)} \right|\leq  \left(\frac{n}{n+k} \right)^i=\left(1-\frac{k}{n+k} \right)^i.
\]
Now define
\begin{align*}
 g(t)= \left(1-\frac{k}{n+k} \right)^{2t}(2n+k)= \left(1+\frac{k}{n} \right)^{-2t}(2n+k).
\end{align*}
Applying~Lemma~\ref{lem:tra} yields,
\begin{align*}
4\| P^t(x,.)-\pi \|^2_{TV} \leq &\sum_{i=1}^n \left(1-\frac{k}{n+k} \right)^{i2t} \cdot \left\{{2n+k \choose i}- {2n+k\choose {i-1}} \right\}\\
\leq &\sum_{i=1}^{n}\frac{\left( (1-\frac{k}{n+k})^{2t}(2n+k) \right)^i}{i!}\\
\leq& e^{g(t)}-1.
\end{align*}
Using the fact that for every $x$, $0\leq x\leq 1/2$, $e^x-1\leq 2x$, we conclude that for any $0\leq g(t)\leq 1/2$,
\begin{align}
\| P^t(x,.)-\pi \|_{TV} \leq \sqrt{g(t)/2}\label{eq:upper}
\end{align}

We consider two cases:\\
\textbf{Case 1.} $k=o(n)$. We choose $t=\frac{1}{2}\log_{1+k/n}(2n+k) +c\frac{n}{k}$, where $c\geq 1/2$. Hence,
\[
 g(t)= \left(1 + \frac{k}{n} \right)^{-2t}(2n+k)=\left(1+ \frac{k}{n} \right)^{-2\frac{cn}{k}}\leq e^{-2c}< 1/2,
\]
and by inequality (\ref{eq:upper}),
\[
d\left(\frac{1}{2}\log_{1+k/n}(2n+k) +c\frac{n}{k}\right)\leq e^{-c}
\]
\textbf{Case 2.} $k = \Omega(n)$. Now we choose $t=\frac{1}{2}\log_{1+k/n}(2n+k)+c$. 
Then, 
\[
 g(t)= \left(1 + \frac{k}{n} \right)^{-2t}(2n+k)=  \left(1+\frac{k}{n} \right)^{-2c}\leq 1/2,
\]
where the last inequality holds due to assumption on $c$.
 Hence, inequality (\ref{eq:upper}) yields
\[
d\left(\frac{1}{2}\log_{1+k/n}(2n+k)+c\right)\leq  \left(1+\frac{k}{n}\right)^{-c}
\]
\end{proof}

\section{Lower Bound on the Variation Distance}
 In order to find a lower bound for variation distance we use  the following lemma which was applied in \cite{wilson}. For further discussion on this method we refer the reader to \cite{laurent}.
Let $f$ be a real-valued function on $\Omega$. We use  $\Exs{\mu}{f}$ and $\Vars{\mu}{f}$ to denote the expectation and variance of $f$ under distribution of $\mu$.

\begin{lem}[{\cite[Proposition~7.8]{lev}}]\label{lem:low} 
Let $\mu$ and $\nu$ be two probability distributions on $\Omega$ and $f:\Omega\rightarrow \mathbb{R}$ be an arbitrary function. Suppose that  $\max\{\Vars{\mu}{f}, \Vars{\nu}{f}\}\leq\sigma^2_{*}$. Then if  
\[
\left|\Exs{\mu}{f}- \Exs{\nu}{f}\right|\geq r\sigma_{*},
\] 
then
\[
\|\mu -\nu \|_{TV}\geq 1-\frac{8}{r^2}.
\] 
\end{lem}

 Before proceeding, we recall that a random variable $Y\sim H(N,m,n)$ has  a  hypergeometric distribution if
 for every  $ \max\{0,n+m-N\}\leq i \leq \min\{n,m\}$,
 $\Pr{Y=i} =\frac{{m\choose i}{{N-m}\choose {n-i}}}{{N\choose n}}$.
The expected value and variance of $Y$ are
$\Ex{Y} =\frac{nm}{N}$ and $\Var{Y}=\frac{nm(N-m)(N-n)}{N^2(N-1)}$ respectively.

\begin{lem}\label{lem:main}
Let $X_t$ be the vertex visited at step $t$ by a simple random walk on $K(2n+k,n)$ which starts at vertex $X_0=\{n+1,n+2,\ldots,2n\}$. Let $f_t=f(X_t)=|X_t\cap[n]|$, so $f_0=0$. Moreover, define a random variable $f=|X\cap[n]|$ with $X$ being a vertex chosen uniformly at random from $K(2n+k,n)$. Then for any $t \in \mathbb{N}$,
\[
\Var{f_t}\leq  C(n,k)\Var{f},
\]
where $C(n,k)=(1+o(1))(1+k/n)$ for $k=O(n)$.
\end{lem}

\begin{proof} The random variable $f$ under $\pi$ has a hypergeometric distribution  $H(2n+k,n,n)$.
Hence,
\begin{align}
 \Ex{f}=\frac{n^2}{2n+k} \label{eq:unex},
\end{align}
and 
\begin{align}
 \Var{f}=\frac{n^2(n+k)^2}{(2n+k)^2(2n+k-1)}\label{eq:univar}.
\end{align}


In step $t+1$ of the walk, an $n$-element subset of the complement of $X_t$ is chosen. If $f_t=s$, $|X_t\cap[n]|=s$, then $X^c_t$ has $n-s$ common  elements with $[n]$ and $s+k$ common elements with $[n]^c$. Therefore $f_{t+1}=n-Y$ where $Y$ has hypergeometric distribution $H(n+k,s+k,n)$. Hence, 
\begin{align*}
  \Ex{f_{t+1} \, \mid \, f_t=s} &= \Ex{n-Y}=n-\frac{(s+k)n}{n+k}=(n-s) \cdot \left(1-\frac{k}{n+k} \right)\\
  &=n \left(1-\frac{k}{n+k} \right)-\Ex{f_t} \left(1-\frac{k}{n+k} \right).
 \end{align*}
Solving this recursion allows us to compute the expectation of $f_t$:
\begin{align}
  \Ex{f_{t}} &= n\sum_{i=1}^{t} \left[(-1)^{i+1} \left(1-\frac{k}{n+k} \right)^i \right]+\underbrace{(-1)^t\Ex{f_0}(1-\frac{k}{n+k})^t}_{=0}\notag\\
&= -n\frac{(\frac{k}{n+k}-1)^{t+1}-(\frac{k}{n+k}-1)}{\frac{k}{n+k}-2}
= \frac{n^2}{2n+k}+(-1)^{t+1}\frac{n(n+k)(1-\frac{k}{n+k})^{t+1}}{2n+k}.\label{eq:exp}
\end{align}
We have already shown that $\Ex{f_{t+1} \, \mid \, f_t}=n(1-\frac{k}{n+k})-{f_t}(1-\frac{k}{n+k})$, which immediately implies that
\begin{align*}
\Var{\Ex{f_{t+1} \, \mid \, f_t}} &= \left(1-\frac{k}{n+k} \right)^2\Var{f_t}.
\end{align*}
As observed earlier, the random variable $f_{t+1}$ conditioned on $f_t$ has distribution $n-Y$ where  $Y\sim H(n+k,f_t+k,n)$ which yields
\begin{align*}
 \Var{f_{t+1} \, \mid \, f_t} &=\Var{n-Y}=\Var{Y}=\frac{(f_t+k)(n-f_t)}{(n+k)^2}\times\frac{nk}{(n+k-1)}.
\end{align*}
Assume now that $A$ is an upper bound for $\frac{(f_t+k)(n-f_t)}{(n+k)^2}$ for every $f_t$; $A$ will be specified later. 
In the following, we  use the total law of variance to find a recursive formula for $\Var{f_t}$,
\begin{align*}
\Var{f_{t+1}} &=\Var{\Ex{f_{t+1} \, \mid \, f_t}}+\Ex{\Var{f_{t+1} \, \mid \, f_t}}\\
&\leq \left(1-\frac{k}{n+k} \right)^2\Var{f_{t}}+A\frac{nk}{(n+k-1)}.
\end{align*}

Using this recursion, we obtain the following upper bound on $\Var{f_t}$:
\begin{align*}
  \Var{f_t} &\leq  A\frac{nk}{n+k-1}\sum_{i=0}^{t-1} \left[ \left(1-\frac{k}{n+k} \right)^{2i} \right] +\underbrace{(1-\frac{k}{n+k})^{2t}V(f_0)}_{=0}\\
&= A\frac{nk}{n+k-1}\times\frac{1-(1-\frac{k}{n+k})^{2t}}{1-(1-\frac{k}{n+k})^2}
\le A\frac{n(n+k)^2}{(2n+k)(n+k-1)}.
\end{align*}

 Since always $0\leq f_t\leq n$, $\frac{(f_t+k)(n-f_t)}{(n+k)^2}\leq 1/4 = A$. 
\[
 \Var{f_t}\leq \frac{1}{4} \cdot \frac{n(n+k)^2}{(2n+k)(n+k-1)}= \frac{1}{4} \cdot \frac{n^3(1+k/n)^2}{n^2(2+k/n)(1+k/n-o(1))}=n\frac{(1+k/n)(1+o(1))}{4(2+k/n)}.
\]
Moreover,
\[
 \Var{f}\geq\frac{n^4(1+k/n)^2}{n^3(2+k/n)^3}.
\]
Using the fact that $1/2\leq \frac{1+x}{2+x}$ for every $x\geq 0$,
\[
\Var{f}\cdot(1+k/n)\cdot(1+o(1))\geq \frac{n(1+k/n)(1+o(1))}{4(2+k/n)}.
\]

By comparing $\Var{f}$ and $\Var{f_t}$, the claim follows.
\end{proof}
We are now ready to apply  Lemma~\ref{lem:low} to derive a lower bound on the total variation distance.
\begin{pro}\label{pro:lower}
 For every constant $c>0$, we have the following lower bounds on the total variation distance for a simple random walk on $K(2n+k,n)$.
\begin{itemize}
 \item If $k=o(n)$, 
\[
 d\left(\frac{1}{2}\log_{1+k/n}(2n+k)-c\frac{n}{k}\right)\geq 1-8(1+o(1))(e-o(1))^{-2c}.
\]
\item If $k=\Theta(n)$,  then
\[
 d\left(\frac{1}{2}\log_{1+k/n}(2n+k)-c\right)\geq 1-8(1+o(1))(1+k/n)^{-2c+4}.
\]
\end{itemize}
 \end{pro}
\begin{proof}

By using Lemma \ref{lem:main} and (\ref{eq:univar})
 \begin{align*}
\sqrt{\max\{\Var{f},\Var{f_t}\}}\leq \sqrt{C(n,k)\Var{f}}\leq C(n,k)\frac{n(n+k)}{(2n+k)\sqrt{2n+k-1}}=\sigma_*.
\end{align*}
Combining~(\ref{eq:exp}) and (\ref{eq:unex}),
\begin{align*}
 |\Ex{f_t}-\Ex{f}| = \frac{n(n+k)}{2n+k} \left(1-\frac{k}{n+k} \right)^{t+1}
=\frac{1}{C(n,k)} \sigma_*\sqrt{2n+k-1}\left(1+\frac{k}{n} \right)^{-t-1}.
\end{align*}
Define 
\[
 \tilde{g}(t)=\frac{\sqrt{2n+k-1}}{C(n,k)} \left(1+\frac{k}{n} \right)^{-t-1}.
\]
\begin{itemize}

 \item {\bf Case 1.} $k=o(n)$. By Lemma~\ref{lem:main} we know that $C(n,k)=(1+k/n)(1+o(1))=(1+o(1))$. We choose $t=\frac{1}{2}\log_{1+k/n}(2n+k)-c\frac{n}{k}$ so that
\[
 \tilde{g}(t)=\frac{\sqrt{1-o(1)}}{1+o(1)} \left(1+\frac{k}{n} \right)^{c\frac{n}{k}}=(1-o(1))e_n^{c},
\]
where $(e_n)_{n}$ is an increasing sequence tending to $e$ as $n \rightarrow \infty$. Applying Lemma~\ref{lem:low} yields,
\[
 d\left(\frac{1}{2}\log_{1+k/n}(2n+k)-c\frac{n}{k}\right) = \| P^t(X_0, .)-\pi\|_{TV} \geq 1- 8(1+o(1))e_n^{-2c},
\]
where $X_0=\{n+1,\ldots,2n\}$ and the equality comes from the fact that the chain is transitive.

\item {\bf Case 2.} $k=\Theta(n)$. By Lemma~\ref{lem:main}, $C(n,k)=(1+k/n)(1+o(1)) $. Take $t=\frac{1}{2}\log_{1+k/n}(2n+k)-c$. Hence,
\[
 \tilde{g}(t)=\frac{\sqrt{1-o(1)}}{1+o(1)} \, (1+k/n)^{c-2}.
\]
Again, using Lemma~\ref{lem:low} gives
\[
 d\left(\frac{1}{2}\log_{1+k/n}(2n+k)-c\right) = \| P^t(X_0,.)-\pi\|_{TV} \geq 1-8(1+o(1)(1+k/n)^{-2c+4}.
\]

\end{itemize}
\end{proof}





\end{document}